\documentclass[11pt,leqno]{amsart}
\usepackage{amsmath,amssymb,amsthm}
\usepackage{hyperref}
\newtheorem{theorem}{Theorem}[section]

\newtheorem{question}{Question}

\newtheorem{proposition}[theorem]{Proposition}
\newtheorem{corollary}[theorem]{Corollary}
\theoremstyle{definition}
\newtheorem{definition}[theorem]{Definition}

\theoremstyle{remark}
\newtheorem{remark}[theorem]{Remark}
\numberwithin{equation}{section}

\def\fnote#1{\footnote}

\def\ignora#1{}
\def\n3#1{\left\vert  \! \left\vert \! \left\vert \, #1 \, \right\vert \!
  \right\vert \! \right\vert }

\begin{document}

\title{ Octahedrality in Lipschitz free Banach spaces }

\author{Julio Becerra Guerrero, Gin{\'e}s L{\'o}pez-P{\'e}rez and Abraham Rueda Zoca}
\address{Universidad de Granada, Facultad de Ciencias.
Departamento de An\'{a}lisis Matem\'{a}tico, 18071-Granada
(Spain)} \email{juliobg@ugr.es, glopezp@ugr.es, arz0001@correo.ugr.es}

\subjclass[2010]{Primary 46B20; Secondary 46B22, 52A10.}

\keywords{Diameter two properties, octahedral norms, slices.}

\maketitle\markboth{J. Becerra, G. L\'{o}pez and A. Rueda}{Octahedrality in Lipschitz free Banach spaces.}

\begin{abstract}
The aim of this note is to study octahedrality in vector valued
Lipschitz-free Banach spaces on a metric space under topological
hypotheses on it by analysing the weak-star strong diameter two property in Lipschitz functions spaces. Also, we show an example which proves that our results are optimal and that octahedrality in vector-valued Lipschitz-free Banach spaces actually relies on the underlying metric space as well as on the Banach one.
\end{abstract}

\section{Introduction}

\bigskip

Lipschitz functions spaces (denoted by $Lip(M)$) and its preduals \cite{wea}, Lipschitz-free Banach spaces (denoted by $\mathcal F(M)$), have been recently studied under a topological point of view (e.g. \cite{dal}, \cite{gokal}, \cite{kal}). Geometrical properties in such spaces have been also considered, as Daugavet property. Indeed,  Daugavet property in Lipschitz functions spaces has been characterized in \cite{ikw} in terms of ``locality'' in the compact case and provides examples of metric spaces whose free-Lipschitz Banach space has an octahedral norm. On the other hand, in \cite{cdw} it has been recently proved that given $M$ an infinite metric space then the free space $\mathcal F(M)$ contains a complemented copy of $\ell_1$ and, consequently, $\mathcal F(M)$ has an equivalent norm which is an octahedral norm \cite{G}.

Motivated by this kind of results, the aim of this note is to go further and analyse octahedrality in vector-valued free-Lipschitz Banach spaces. Indeed, we introduce the Banach space  of vector-valued Lipschitz-free Banach space which, up the best of our knowledge, has not been previously considered and, in Section \ref{mainresults}, we prove in Theorem \ref{teocentral} that such spaces have an octahedral norm whenever their underlying metric space satisfies some topological assumptions such as being unbounded or not being uniformly discrete and a condition of existence of extension of vector-valued Lipschitz functions (see Definition \ref{defiCPE}). Consequently, such Banach spaces can not have any point of Fr\'echet differentiability. Moreover, we will exhibit an example of a metric space such that, depending on the underlying Banach space, the geometry of vector valued Lipschitz-free Banach space changes its behaviour from having a point of Fr\'echet differentiability to having an octahedral norm. This will have two important consequences: on the one hand, as there are vector-valued Lipschitz-free Banach spaces which contain points of Fr\'echet differentiability, we prove that our results on octahedrality are optimal; on the other hand, this proves that the geometry of vector-valued Lipschitz-free Banach spaces is determined by the underlying metric space as well as by the target Banach space. We will end by exhibiting some consequences of Theorem \ref{teocentral} and open problems in Section \ref{conseopenquestions}

We shall now introduce some notation. We will consider only real Banach spaces. Given $X$ a Banach space, $B_X$ (respectively $S_X$) stands for the closed unit ball (respectively the unit sphere) of $X$. Given a Banach space $X$, we will mean by a slice of $B_X$ a subset of the following form
$$S(B_X,f,\alpha):=\{x\in B_X: f(x)>1-\alpha\},$$
where $f\in S_{X^*}$ and $\alpha>0$. If $X$ is a dual space, say $X=Y^*$, by a weak-star slice of $B_{X^*}$ we will mean a slice $S(B_X,y,\alpha)$ where $y\in Y$.

Given $M$ a metric space with a designated origin $0$ and a Banach space $X$, we will denote by $Lip(M,X)$ the Banach space of all $X$-valued Lipschitz function on $M$ which vanish at $0$ under the standard Lipschitz norm
$$\Vert f\Vert:=\sup\left\{\left. \frac{\Vert f(x)-f(y)\Vert}{d(x,y)}\ \right/\ x,y\in M, x\neq y \right\} .$$
First of all, notice that we can consider every point of $M$ as an origin with no loss of generality. Indeed, given $x,y\in M$, let $Lip_x(M,X)$ ($Lip_y(M,X)$) be the space of $X$-valued Lipschitz functions which vanish at $x$ (respectively at $y$). Then the map
$$\begin{array}{ccc}
Lip_x(M,X) & \longrightarrow & Lip_y(M,X)\\
f & \longmapsto & f-f(y),
\end{array}
$$
defines an onto linear isometry. So the designated origin will be freely chosen.

From a straightforward application of Ascoli-Arzela theorem it can be checked that $B_{Lip(M,X^*)}$ is a compact set for the pointwise topology. Hence $Lip(M,X^*)$ is itself a dual Banach space. In fact, the map
$$\begin{array}{ccc}
\delta_{m,x}:Lip(M,X^*) & \longrightarrow & \mathbb R\\
f & \longmapsto & f(m)(x)
\end{array}$$
defines a linear and bounded map for each $m\in M$ and $x\in X$. In other words, $\delta_{m,x}\in Lip(M,X^*)^*$. Then if we define
$$\mathcal F(M,X):=\overline{span}(\{\delta_{m,x}\ /\ m\in M, x\in X\})$$
then we have that $\mathcal F(M,X)^*=Lip(M,X^*)$. Furthermore, a bounded net $\{f_s\}$ in $Lip(M,X^*)$ converges in the weak-star topology to a function $f\in Lip(M,X^*)$ if, and only if, $\{f_s(m)\}\rightarrow f(m)$ for each $m\in M$, where last convergence is in the weak-star topology of $X^*$. 

Note that given $f:M\longrightarrow X^*$ a Lipschitz map then there exists a linear operator $T_f:\mathcal F(M)\longrightarrow X^*$ such that $T_f\circ \delta_m=f(m)$ for each $m\in M$. This map

$$\begin{array}{ccc}
\Phi:Lip(M,X^*) & \longrightarrow & L(\mathcal F(M),X^*)\\
f & \longmapsto & T_f
\end{array}$$
is an isometric isomorphism (see e.g. \cite{jsv}). Now we have the following

\begin{proposition}\label{w*-w*continuidad}
$\Phi$ is $w^*-w^*$ continuous.
\end{proposition}

\begin{proof} Note that $\Phi$ is $w^*-w^*$ continuous if, and only if, for every $z\in \mathcal F(M)\widehat{\otimes}_\pi X$ one has that $z\circ \Phi$ is a weak-star continuous functional. By \cite[Corollary 3.94]{fab} it is enough to prove that, given $z\in \mathcal F(M)\widehat{\otimes}_\pi X$, we have that $Ker(z\circ \Phi)\cap B_{Lip(M,X^*)}$ is weak-star closed. So, pick $z\in \mathcal F(M)\widehat{\otimes}_\pi X$ and consider $\{f_s\}$ a net in $Ker(z\circ \Phi)\cap B_{Lip(M,X^*)}$ which is weak-star convergent to $f$, and let us prove that $(z\circ \Phi)(f)=0$. To this aim, pick a positive number $\varepsilon>0$. Note that $z$ can be expressed as 
$$z:=\sum_{n=1}^\infty \gamma_n\otimes x_n$$
where $\gamma_n\in \mathcal F(M)$ and $x_n\in X$ verify that $\sum_{n=1}^\infty \Vert \gamma_n\Vert\Vert x_n\Vert<\infty$ \cite[Proposition 2.8]{rya}. Now, consider $\{\varepsilon_n\}$ a sequence in $\mathbb R^+$ such that $\sum_{n=1}^\infty \varepsilon_n<\frac{\varepsilon}{3}$ and consider, for each $n\in\mathbb N$, an element $\psi_n\in span\{\delta_m:m\in M\}$ verifying $\Vert \gamma_n-\psi_n\Vert\Vert x_n\Vert<\frac{\varepsilon_n}{2}$ for each $n\in\mathbb N$. As it is clear that $\sum_{n=1}^\infty \Vert \psi_n\Vert\Vert x_n\Vert<\infty$, consider $k\in\mathbb N$ such that $\sum_{n=k+1}^\infty \Vert \psi_n\Vert\Vert x_n\Vert<\frac{\varepsilon}{6}$. Finally, in view of weak-star topology of $Lip(M,X^*)$, it is obvious that $\{f_s(\psi_n)(x_n)\}\rightarrow f(\psi_n)(x_n)$ for each $n\in\mathbb N$, hence we can find $s$ such that $\vert (f-f_s)(\psi_n)(x_n)\vert<\frac{\varepsilon}{3k}$ for each $n\in\{1,\ldots, k\}$. Now, keeping in mind that $\Vert f-f_s\Vert\leq 2$, one has
$$\vert (z\circ \Phi)(f)\vert=\vert (z\circ \Phi)(f-f_s)\vert=\left \vert \sum_{n=1}^\infty T_{f-f_s}(\gamma_n)(x_n)\right\vert\leq$$
$$ \left\vert \sum_{n=1}^\infty T_{f-f_s}(\psi_n)(x_n)\right\vert+\Vert f-f_s\Vert\sum_{n=1}^\infty \Vert \gamma_n-\psi_n\Vert\Vert x_n\Vert\leq $$
$$ \sum_{n=1}^k \vert (f-f_s)(\psi_n)(x_n)\vert+\Vert f-f_s\Vert \sum_{n=k+1}^\infty \Vert \psi_n\Vert\Vert x_n\Vert+\frac{\varepsilon}{3}$$
$$<\sum_{n=1}^k \frac{\varepsilon}{3k}+\frac{2\varepsilon}{3}=\varepsilon.$$
As $\varepsilon>0$ was arbitrary we conclude that $(z\circ \Phi)(f)=0$, so we are done.
\end{proof}

The norm on a Banach space $X$ is said to be  octahedral if for every $\varepsilon>0$ and for every finite-dimensional subspace $M$ of $X$ there is some $y$ in the unit sphere of $X$ such that $$\Vert x+\lambda y\Vert\geq (1-\varepsilon)(\Vert x\Vert +\vert \lambda\vert)$$ holds for every $x\in M$ and for every scalar $\lambda$ (see \cite{dgz}).

We recall that a Banach space $X$ satisfies the slice diameter two property (respectively diameter two property, strong diameter two property) if every slice (respectively nonempty weakly open subset, convex combination of slices) of the closed unit ball has diameter two. If $X$ is itself a dual Banach space then weak-star slice diameter two property (respectively weak-star diameter two property and weak-star strong diameter two property) can be defined as usual, invoking weak-star slices (respectively nonempty weakly-star open subset, convex combination of weak-star slices) of the unit ball of $X$. These property, which are extremely opposite to Radon-Nikod\'ym property, have been deeply studied since last few years. For instance, it has been recently proved \cite{primeje, adv} that each one of the above properties is different from  the rest in an extreme way.

Banach spaces which satisfy some of the diameter two properties are infinite-dimensional uniform algebras \cite{nywe}, Banach spaces satisfying Daugavet property \cite{sh},  non-reflexive M-embedded spaces \cite{gines}, etc.

It is known that the norm on a Banach space $X$ is octahedral if, and only if, $X^*$ satisfies the weak-star strong diameter two propery \cite[Theorem 2.1]{blr}. It is also known that if a Banach space $X$ has an octahedral norm, then $X$ contains an isomorphic copy of $\ell_1$ \cite{G}.

Finally, given a Banach space $X$ and a point $x\in X$, we say that $x$ is a point of Fr\'echet differentiability if, for each $h\in X$, we have that
$$f'(x)(h):=\lim\limits_{t\rightarrow 0} \frac{\Vert x+th\Vert-\Vert x\Vert}{t},$$
  uniformly for  $h\in S_X$ and $f'(x):X\longrightarrow \mathbb R$ is a continuous and linear functional (see \cite{G}).

It is clear that a Banach space which has an octahedral norm does not have any point of Fr\'echet differentiability.

\section{Main results.}\label{mainresults}

\bigskip

Let $M$ be a metric space and $X$ a Banach space. Notice that we have a useful description of $\mathcal F(M,X)$ because we know a dense subspace of it. This fact will play an important role in the following because diameter two properties actually rely on dense subspaces in the following sense.

\begin{proposition}\label{densid2pdual}

Let $X$ be a Banach space. Let $Y\subseteq X^*$ a norm dense subspace. Then:

\begin{enumerate}
\item If for each $f\in S_Y$ and $ \alpha\in\mathbb R^+$ the slice $S(B_X,f,\alpha)$ has diameter two, then $X$ has the slice diameter two property.

\item If for each $f_1,\ldots, f_n\in S_Y$ and $ \alpha_1,\ldots, \alpha_n\in\mathbb R^+$ such that  $W:=\bigcap\limits_{i=1}^n S(B_X,f_i,\alpha_i)\neq \emptyset$ it follows that $W$ has diameter two, then $X$ has the diameter two property.

\item If for each $f_1,\ldots, f_n\in S_Y, \alpha_1,\ldots, \alpha_n\in\mathbb R^+$ and  $\lambda_1,\ldots, \lambda_n\in ]0,1]$ with $\sum_{i=1}^n \lambda_i=1$, the convex combination of slices\break $\sum_{i=1}^n \lambda_i S(B_X,f_i,\alpha_i)$ has diameter two, then $X$ satisfies the strong diameter two property
\end{enumerate}

\end{proposition}

\begin{proof}
We will prove statement (1), being the proof of (2) and (3) completely similar.

Pick $S:=S(B_X,f,\alpha)$ a slice of $B_X$. As $Y$ is norm dense in $X^*$ we can find $\varphi\in S_Y$ such that $\Vert f-\varphi\Vert<\frac{\alpha}{2}$.

By hypothesis, given an arbitrary $\delta\in\mathbb R^+$  we can find $x,y\in S(B_X,\varphi,\frac{\alpha}{2})$ such that $\Vert x-y\Vert>2-\delta$. Let us prove that $x\in S$, being the proof of $y\in S$ similar. Bearing in mind that $\varphi(x)>1-\frac{\alpha}{2}$ and that $\Vert f-\varphi\Vert<\frac{\alpha}{2}$ we deduce
$$f(x)=\varphi(x)+(f-\varphi)(x)\geq \varphi(x)-\Vert f-\varphi\Vert >1-\alpha.$$
On the other hand, as $x,y\in S$, we conclude
$$2-\delta<\Vert x-y\Vert\leq diam(S).$$
As $\delta\in\mathbb R^+$ was arbitrary we conclude that $X$ has the slice diameter two property, as desired\end{proof}

Now we consider the weak-star version of Proposition above.

\begin{proposition}\label{densid2ppredual}

Let $X$ be a Banach space and $Y\subseteq X$ be a dense subspace. Then:

\begin{enumerate}
\item If for each $y\in S_Y$ and $ \alpha\in\mathbb R^+$ the slice $S(B_{X^*},y,\alpha)$ has diameter two, then $X$ has the weak-star slice diameter two property.

\item If for each $y_1,\ldots, y_n\in S_Y$ and $ \alpha_1,\ldots, \alpha_n\in\mathbb R^+$ such that  $W:=\bigcap\limits_{i=1}^n S(B_{X^*},y_i,\alpha_i)\neq \emptyset$ one has that $W$ has diameter two, then $X$ has the weak-star diameter two property.

\item If for $y_1,\ldots, y_n\in S_Y, \alpha_1,\ldots, \alpha_n\in\mathbb R^+$ and $\lambda_1,\ldots, \lambda_n\in ]0,1]$ with  $\sum_{i=1}^n \lambda_i=1$ the convex combination of weak-star slices\break $\sum_{i=1}^n \lambda_i S(B_{X^*},y_i,\alpha_i)$ has diameter two, then $X$ satisfies the strong diameter two property
\end{enumerate}

\end{proposition}

Now we need the following

\begin{definition}\label{defiCPE}
Let $M$ be a metric space and let $X$ be a Banach space.

We will say that the pair $(M,X)$ satisfies the contraction-extension property (CPE) if given $N\subseteq M$ and $f:N\longrightarrow X$ a Lipschitz function then there exists $F:M\longrightarrow X$ a Lipschitz function which extends to $f$ such that
$$\Vert F\Vert_{Lip(M,X)}=\Vert f\Vert_{Lip(N,X)}.$$
\end{definition}

On the one hand note that, in the particular case of being $M$ a Banach space, the definition given above agrees with the one given in \cite{beli}.

On the other hand, let us give some examples of pairs which  have the CPE. First of all, given $M$ a metric space, the pair $(M,\mathbb R)$ has the (CPE) \cite[Theorem 1.5.6]{wea}. In addition, in \cite[Chapter 2]{beli} we can find some examples of Banach spaces $X$ such that the pair $(X,X)$ satisfies the contraction extension property as Hilbert spaces and $\ell_\infty^n$. Finally, if $Y$ is a strictly convex Banach space such that there exists a Banach space $X$ with $dim(X)\geq 2$ and verifying that the pair $(X,Y)$ has the CPE, then $Y$ is a Hilbert space \cite[Theorem 2.11]{beli}.

Let us explain roughly the key idea of the main result, which proves, for every unbounded or not uniformly discrete metric space $M$, that the norm on $\mathcal F(M,X)$ is octahedral, whenever the pair $(M,X^*)$ has the CPE, where $X$ is a Banach space. For this, it is enough to show that every convex combination of $w^*-$slices C in the unit ball of $Lip(M,X^*)$ has diameter exactly $2$. What it is done first is to observe that it is enough to  consider $w^*$-slices given by elements in $span\{\delta_{m,x}\ /\ m\in M, x\in X\}$, which is based on Proposition \ref{densid2ppredual}. Now, depending on the kind of considered metric space, we construct a pair of Lipschitz functions for every $w^*-slice$ of $C$. Different pairs are defined on different finite metric subspaces so that each pair of these functions have norm preserving extensions to Lipschitz functions in a $w^*-$slice of $C$ from the CPE assumption, and we control the norm of each pair only on a finite metric space, so that the difference between the elements of every pair is also controlled. Now the estimates for the extensions are possible and in this way we get that $C$ has diameter $2$. Of course, there are details which depend on the kind of considered metric space, but the existence of the above unified idea motivated to us to show the following result in a joint way.    

\begin{theorem}\label{teocentral}
Let $M$ be an infinite pointed metric space and let $X$ be a Banach space. Assume that the pair $(M,X^*)$ has the CPE. If $M$ is unbounded or is not uniformly discrete then the norm on $\mathcal F(M,X)$ is octahedral. Consequently, the unit ball of $\mathcal F(M,X)$ can not have any point of Fr\'echet differentiability.
\end{theorem}

\begin{proof} We will prove that $Lip(M,X^*)$ has the weak-star strong diameter two property, which is equivalent to the thesis of the Theorem. Let $C=\sum_{i=1}^k \lambda_i S(B_{Lip(M,X^*)},\varphi_i, \alpha)$ be a convex combination of weak-star slices in $Lip(M,X^*)$ and let us prove that $C$ has diameter exactly $2$.  From Proposition \ref{densid2pdual} we can assume that $\varphi_i\in span\{\delta_{m,x}\ /\ m\in M, x\in X\}$ for each $i\in\{1,\ldots, k\}$. So assume that
$$\varphi_i=\sum_{j=1}^{n_i}\lambda_j^i \delta_{m_{i,j},x_{i,j}},$$
for suitable $n_i\in\mathbb N$, $m_{i,j}\in M\setminus\{0\}, x_{i,j}\in X,  \lambda_j^i\in\mathbb R$ for $i\in\{1,\ldots, k\}, j\in\{1,\ldots, n_i\}$.

 Pick $g_i\in S(B_{Lip(M,X^*)},\varphi_i,\alpha)$ and $\delta_0\in\mathbb R^+$ verifying
$$0<\delta<\delta_0\Rightarrow \frac{\varphi_i(g_i)}{1+\delta}>1-\alpha\ \ \forall i\in\{1,\ldots, k\}.$$
Fix $0<\delta<\delta_0$.

Now, in a first step we will define for every $i\in\{1,\ldots ,k\}$ a subspace $M_i\subset M$ and functions $F_i$ and $G_i$ in $Lip(M_i,X^*)$. 
 
We will do this depending on following cases: $M$ is unbounded, $M$ is bounded, discrete but not uniformly discrete or $M$ is bounded and $0\in M'$. It is clear that each one of these cases holds whenever $M$ is unbounded or not uniformly discrete and that all together cover the assumption of the Theorem.

Assume that $M$ is unbounded. Then there exists $\{m_n\}\subseteq M$ verifying
$$\{d(m_n,0)\}\rightarrow\infty.$$
Hence
$$\{d(m_n,m)\}\rightarrow \infty$$
for each $m\in M$ in view of triangle inequality. Now pick an positive integer $N$ so that
\begin{equation}\label{condinatuepsilon}
 \frac{d(m_{i,j},0)}{d(m_N,m_{i,j})}+\frac{\Vert g_i(m_{i,j})\Vert}{d(m_N,m_{i,j})}<\delta\ \ \forall i\in\{1,\ldots, k\}, j\in\{1,\ldots, n_i\}.
\end{equation}
Choose $x^*\in S_{X^*}$ and define $M_i:=F:=\{0\}\cup\bigcup\limits_{i=1}^k \bigcup\limits_{j=1}^{n_i} \{m_{i,j}\}\cup \{m_N\}$ for every $1\leq i\leq k$. (In this case $M_i$ does not depend on $i$). Also, we define $F_i,G_i:M_i\longrightarrow X^*$ given by
$$F_i(m_{i,j})=G_i(m_{i,j})=
g_i(m_{i,j})\ \ i\in\{1,\ldots, k\}, j\in\{1,\ldots,n_i\},$$
$$F_i(0)=G_i(0)=0, F_i(m_N)=-G_i(m_N)=d(m_N,0)x^*. $$

Assume now that $M$ is bounded and discrete, but not uniformly discrete. As $M$ is discrete  we can find $r>0$ such that
$$B(0,r)=\{0\}, B(m_{i,j},r)=\{m_{i,j}\}\ \forall i\in\{1,\ldots, k\}, j\in\{1,\ldots, n_i\}.$$
Also, as $M$ is not uniformly discrete  we can find $\{x_n\},\{y_n\}$ a pair of sequences in $M$ such that $0<d(x_n,y_n)\rightarrow 0$. Pick  $n\in\mathbb N$ satisfying  $d(x_n,y_n)<\delta$ and so that
\begin{equation}\label{condiepsilon}
\frac{1+\frac{d(x_n,y_n)}{d(x_n,v)}}{1-\frac{d(x_n,y_n)}{d(x_n,v)}}<1+\delta\ \ \forall v\in\{m_{i,j}: 1\leq i\leq k, 1\leq j\leq n_i\}\cup\{0\}.
\end{equation}
Note that such an $n$ exists since $\{d(x_n,v)^{-1}\}$ is a well defined bounded sequence because $M$ is discrete and bounded in this case. Given $i\in\{1,\ldots, k\}$ and $x^*\in S_{X^*}$ define $M_i:=\{0\}\cup \bigcup\limits_{j=1}^{n_i}\{m_{i,j}\}\cup \{x_n,y_n\}$ and $F_i,G_i:=M_i\longrightarrow \mathbb R$ given by
$$F_i(0)=g_i(0)=0, F_i(m_{i,j})=G_i(m_{i,j})=g_i(m_{i,j})\ \forall j\in\{1,\ldots, n_i\},$$
and
$$ F_i(x_n)=G_i(x_n)=g_i(x_n), F_i(y_n)=g_i(x_n)+d(y_n,x_n)x^*,$$
$$ G_i(y_n)=g_i(x_n)-d(y_n,x_n)x^*.$$

Finally, we assume that $M$ is bounded and $0\in M'$. Then we can find $\{m_n\}$ a sequence in $M\setminus\{0\}$ such that $\{m_n\}\rightarrow 0$. So there exists a positive integer $m$ such that $m_n\notin \{m_{i,j}: i\in\{1,\ldots, k\}, j\in \{1,\ldots, n_i\}\}$ for every $n\geq m$. Now pick $x^*\in S_{X^*}$ and, for each $i\in \{1,\ldots, k\}$, we define $M_{i}:=\{0,m_n\}\bigcup \cup_{j=1}^{n_i}\{m_{i,j}\}$ and $F_i, G_i:M_{i}\longrightarrow X^*$ by the equations

$$F_i(m_{i,j})=G_i(m_{i,j})=g_i(m_{i,j})\ \ i\in\{1,\ldots, k\}, j\in\{1,\ldots, N_i\}$$

and

$$F_i(m_n)=-G_i(m_n)=d(m_n,0)x^*, F_i(0)=G_i(0)=0.$$

Now, for each unbounded or not uniformly discrete metric space $M$ we have defined the desired subspaces $M_i$ and functions $F_i$ and $G_i$ in $Lip(M_i,X^*)$ for every $1\leq i\leq k$.

For a second step we claim that $\Vert F_i\Vert_{Lip(M_i,X^*)}\leq 1+\delta$ for all $i\in\{1,\ldots, k\}$. For this we have three cases again:  $M$ is unbounded, $M$ is bounded, discrete but not uniformly discrete or $M$ is bounded and $0\in M'$. 

Assume that $M$ is unbounded. Given $u,v\in F, u\neq v$ and  $i\in\{1,\ldots, k\}$ we have two different possibilities:
\begin{enumerate}
\item[a)] If $u,v\notin\{m_N\}$ then
$$\frac{\Vert F_i(u)-F_i(v)\Vert}{d(u,v)}=\frac{\Vert g_i(u)-g_i(v)\Vert}{d(u,v)}\leq \Vert g_i\Vert\leq 1.$$
\item[b)] If $u=m_N$ then
$$\frac{\Vert F_i(u)-F_i(v)\Vert}{d(u,v)}=\frac{\Vert d(m_N,0)x^*-F_i(v)\Vert}{d(m_N,v)}\leq \frac{d(m_N,0)}{d(m_N,v)}+\frac{\Vert g_i(v)\Vert}{d(m_N,v)}\leq$$
$$\leq 1+\frac{d(v,0)}{d(m_N,v)}+\frac{\Vert g_i(v)\Vert}{d(m_N,v)}\mathop{<}\limits^{\mbox{
(\ref{condinatuepsilon})}}1+\delta.$$
\end{enumerate}
Now, taking supremun in $u$ and $v$, one has
$$\left\Vert F_i\right\Vert_{Lip(M_i,X^*)}\leq 1+\delta.$$

Assume now that $M$ is bounded, discrete but not uniformly discrete. Again given $u,v\in M_i, u\neq v$ and  $i\in\{1,\ldots, k\}$  we have different possibilities:\begin{enumerate}
\item[a)] If $u\neq y_n$ and $v\neq y_n$ then we have
$$\frac{\Vert F_i(u)-F_i(v)\Vert}{d(u,v)}=\frac{\Vert g_i(u)-g_i(v)\Vert}{d(u,v)}\leq \Vert g_i\Vert\leq 1.$$
\item[b)] If $u=y_n$, $v\neq x_n$ then
$$\frac{\Vert F_i(u)-F_i(v)\Vert}{d(u,v)}=\frac{\Vert g_i(x_n)+d(x_n,y_n)x^*-g_i(v)\Vert}{d(y_n,v)}\leq \frac{\Vert g_i(x_n)-g_i(v)\Vert+d(x_n,y_n)}{d(y_n,v)}\leq$$
$$<\frac{d(x_n,v)+d(x_n,y_n)}{d(x_{n},v)-d(y_n,x_n)}=\frac{1+\frac{d(x_n,y_n)}{d(x_n,v)}}{1-\frac{d(x_n,y_n)}{d(x_n,v)}}<1+\delta.$$
\item[c)] If $u=y_n$ and $v=x_n$ then
$$\frac{\Vert F_i(u)-G_i(v)\Vert}{d(u,v)}=\frac{d(x_n,y_n)\Vert x^*\Vert}{d(x_n,y_n)}=1$$
\end{enumerate}
Then, taking supremum in $u$ and $v$, one has
$$\Vert F_i\Vert_{Lip(M_i,X^*)}\leq 1+\delta.$$

If $M$ is bounded and $0\in M'$ we can get also that  $$\Vert F_i\Vert_{Lip(M_i,X^*)}\leq 1+\delta$$ using similar arguments to the ones of the above case taking $n$ large enough.

Similar computations also arise
$$\Vert G_i\Vert_{Lip(M_i,X^*)}\leq 1+\delta\ \forall i\in \{1,\ldots, k\}.$$

Now, we have defined subspaces $M_i\subset M$ and functions $F_i,\ G_i\in Lip(M_i,X^*)$ such that 
$$\max\limits_{1\leq i\leq k}\{\Vert F_i\Vert_{Lip(M_i,X^*)},\Vert G_i\Vert_{Lip(M_i,X^*)}\}\leq 1+\delta.$$ 
As the pair $(M,X^*)$ has the CPE then, for each $i\in\{1,\ldots, k\}$, we can find   an extension of $F_i$ and $G_i$ to the whole $M$ respectively, which we will call again $F_i$ and $G_i$, respectively, such that
$$\left\Vert F_i\right\Vert_{Lip(M,X^*)}\leq 1+\delta,  \left\Vert G_i\right\Vert_{Lip(M,X^*)}\leq 1+\delta.$$
So $\frac{F_i}{1+\delta}, \frac{G_i}{1+\delta}\in B_{Lip(M,X^*)}$ for each $i\in\{1,\ldots, k\}$. 

The final step of the proof is to see that $\sum_{i=1}^{k}\frac{F_i}{1+\delta}\in C$, $\sum_{i=1}^{k}\frac{G_i}{1+\delta}\in C$ and to conclude from here that $C$ has diameter 2. We prove this fact in the case $M$ is unbounded. For the other cases, the arguments and estimates are similar.
Then, assume $M$ is unbounded. Given $i\in\{1,\ldots, k\}$ one has
$$\varphi_i\left(\frac{F_i}{1+\delta}\right)=\frac{\sum_{j=1}^{n_i}
\lambda_j^i F_i(m_{i,j})(x_{i,j})}{1+\delta}=\frac{\sum_{j=1}^{n_i}
\lambda_j^i g_i(m_{i,j})(x_{i,j})}{1+\delta}=\frac{g_i(\varphi_i)}{1+\delta}>1-\alpha.$$
So $\sum_{i=1}^k \lambda_i \frac{F_i}{1+\delta}\in C$. Similarly one has $\sum_{i=1}^k \lambda_i \frac{G_i}{1+\delta}\in C$. Hence
$$diam(C)\geq \left\Vert \sum_{i=1}^k \lambda_i \frac{F_i}{1+\delta}-\sum_{i=1}^k \lambda_i \frac{G_i}{1+\delta} \right\Vert\geq \frac{\left\Vert \sum_{i=1}^k \lambda_i \frac{F_i(m_N)}{1+\delta}-\sum_{i=1}^k \lambda_i \frac{G_i(m_N)}{1+\delta} \right\Vert}{d(m_N,0)}$$
$$=\frac{\left\Vert \sum_{i=1}^k 2\lambda_i \frac{d(m_N,0)x^*}{1+\delta} \right\Vert}{d(m_N,0)}=\frac{2}{1+\delta}.$$
From the estimate above we deduce that $diam(C)=2$ from the arbitrariness of $0<\delta<\delta_0$.
\end{proof}

Now let us end the section by analysing the vector-valued Lipschitz-free Banach space over a concrete metric space. From here, we will get two interesting consequences: on the one hand, we will get several examples of vector-valued Lipschitz-free Banach spaces which not only fail to have an octahedral norm but also its unit ball contains points of Fr\'echet differentiability. On the other hand, we will prove that such construction depends strongly on the underlying target Banach space. So, octahedrality in vector-valued Lipschitz-free Banach spaces relies on the underlying metric spaces as well as on the target Banach one.

For the construction of such a metric space consider $\Gamma$ to be an infinite set. Define  $M:=\Gamma\cup\{0\}\cup \{z\}$. Consider on $M$ the following distance:
$$d(x,y):=\left\{\begin{array}{cc}
1 & \mbox{if }x,y\in \Gamma\cup\{0\}, x\neq y,\\
1 & \mbox{if }x=z, y\in \Gamma \mbox{ or } x\in \Gamma, y=z,\\
2 & \mbox{if }x=z, y=0\mbox{ or }x=0,y=z,\\
0 & \mbox{Otherwise.}
\end{array} \right.$$

This is obviously an infinite, bounded and uniformly discrete metric space. Moreover, it is not difficult to prove that the pair $(M,X)$ has the CPE for every Banach space $X$. Consider a Banach space $X$, pick $y\in S_X$ and notice that $\delta_{z,y}$ is a 2-norm functional, so define $\varphi:=\frac{\delta_{z,y}}{2}\in S_{\mathcal F(M,X)}$. Given $\alpha\in\mathbb R^+$ consider
$$S_\alpha:=S\left(B_{Lip(M,X^*)},\varphi,\frac{\alpha}{2}\right)=\{f\in B_{Lip(M,X^*)}\ /\ f(z)(y)>2-\alpha\}.$$

Consider $x\in \Gamma$ and $f\in S_\alpha$. We claim that 
$$f(x)(y)>1-\alpha.$$

Indeed, assume by contradiction that $f(x)(y)\leq 1-\alpha$. Then
$$1<f(z)(y)-f(x)(y)=(f(z)-f(x))(y)\leq \Vert f(z)-f(x)\Vert\leq d(z,x)=1,$$
a contradiction.

We will prove that $\inf_\alpha diam(S_\alpha)$ depends on the target space $X^*$.

\begin{proposition}\label{metriraropuntodife}
If $y$ is a point of Fr\'echet differentiability of $B_X$, then $\inf_\alpha S_\alpha=0$.
\end{proposition}

\begin{proof}
Notice that, as $y$ is a point of Fr\'echet differentiability, then there exists (Smulian lemma) $\delta:\mathbb R^+\longrightarrow \mathbb R^+$ such that $\delta(\varepsilon)
\mathop{\longrightarrow}\limits^{\varepsilon\rightarrow 0} 0$ and such that 
\begin{equation}\label{condifrechet}
\left. \begin{array}{c}
x^*,y^*\in B_{X^*}\\
x^*(y)>1-\alpha\\
y^*(y)>1-\alpha
\end{array} \right\}\Rightarrow \Vert x^*-y^*\Vert<\delta(\alpha).
\end{equation}
Pick $f,g\in S\left(B_{Lip(M,X^*)}, \varphi,\frac{\alpha}{2}\right)$ and  $u,v\in M\setminus\{0\}, u\neq v$. Our aim is to estimate
$$\frac{\Vert f(u)-g(u)-(f(v)-g(v))\Vert}{d(u,v)}\leq \Vert f(u)-g(u)-(f(v)-g(v))\Vert\leq$$
$$\leq \Vert f(u)-g(u)\Vert+\Vert f(v)-g(v)\Vert=:K.$$
If $u=z$ then we have
$$\frac{f(u)(y)}{2}>1-\frac{\alpha}{2},\frac{g(u)(y)}{2}>1-\frac{\alpha}{2}\mathop{\Longrightarrow}
\limits^{(\ref{condifrechet})}\Vert f(u)-g(u)\Vert\leq 2 \delta\left( \frac{\alpha}{2}\right).$$
Similarly, if $u\in \Gamma$ then
$$f(u)(y)>1-\alpha, g(u)(y)>1-\alpha\mathop{\Longrightarrow}
\limits^{(\ref{condifrechet})}\Vert f(u)-g(u)\Vert\leq \delta\left( \alpha\right).$$
Hence $K\leq \delta(\alpha)+\max\left\{\delta(\alpha),
2\delta\left(\frac{\alpha}{2}\right) \right\}$.

From the arbitrariness of $f,g\in S\left(B_{Lip(M)}, \varphi,\frac{\alpha}{2}\right)$ we conclude that 
$$diam\left(S\left(B_{Lip(M)}, \varphi,\frac{\alpha}{2}\right)\right)\leq \delta(\alpha)+\max\left\{\delta(\alpha),2
\delta\left(\frac{\alpha}{2}\right) \right\}.$$

Finally, taking infimum in $\alpha\in\mathbb R^+$, from the hypothesis on $\delta$ and the continuity of the map $\max$ we conclude the desired result.\end{proof}

Despite above Proposition, we will prove that $\mathcal F(M,X)$ has a dramatically different behaviour whenever $X^*$ has the weak-star slice diameter two property.

\begin{proposition}\label{metrirarodiam2}
If $X^*$ has the weak-star slice diameter two property, then $\inf_\alpha S_\alpha=2$.
\end{proposition}

\begin{proof}
Pick two arbitrary numbers $\alpha>0$ and $\varepsilon>0$. As $X^*$ has the weak-star slice diameter two property we can find $x^*,y^*\in S\left(B_{X^*},x,\frac{\alpha}{2}\right)$ such that $\Vert x^*-y^*\Vert>2-\varepsilon$.  Now define $f,g: M\longrightarrow X^*$ by the equations
$$f(t):=d(t,0)x^*\  \ \ g(t):=d(t,0)y^*\ \ \forall t\in M.$$
Now $f,g$ are clearly norm one Lipschitz functions. Moreover
$$\varphi(f)=\frac{f(z)(x)}{2}=x^*(x)>1-\frac{\alpha}{2}.$$
So $f\in S_\alpha$. Analogously $g\in S_\alpha$. Consequently
$$diam(S_\alpha)\geq \Vert f-g\Vert\geq \frac{\Vert f(z)-g(z)\Vert}{2}=\Vert x^*-y^*\Vert>2-\varepsilon.$$
As $\varepsilon$ and $\alpha$ were arbitrary we conclude that $diam(S_\alpha)=2$, so we are done.
\end{proof}

From two Propositions above we can get the desired consequences.  From Proposition \ref{metriraropuntodife} we get vector-valued Lipschitz-free Banach spaces with points of Fr\'echet differentiability which, keeping in mind that the pair $(M,X^*)$ has the (CPE) for every Banach space $X$, proves that Theorem \ref{teocentral} is optimal. However, from Proposition \ref{metrirarodiam2} we conclude that the existence of such Fr\'echet differentiability point depends on the target space. Indeed, we can even get octahedrality for suitable choices of $X$ in example above. For instance, $\mathcal F(M,\ell_1)=\mathcal F(M)\widehat{\otimes}_\pi \ell_1=\ell_1(\mathcal F(M))$ has an octahedral norm.

\section{Consequences and open questions.}\label{conseopenquestions}

Under the assumptions of Theorem \ref{teocentral} we have that $Lip(M,X^*)$ has the weak-star strong diameter two property. This arises a natural question.

\begin{question}
Let $M$ and $X$ under the hypothesis of Theorem \ref{teocentral}.
Does $Lip(M,X^*)$ satisfy the strong diameter two property?
\end{question}

Note that we have a partial answer in \cite{ikw} for the scalar case in terms of Daugavet property. Also, in the scalar case, when $M$ is a compact metric space such that $lip(M)$ separates the points in $M$, it is known that $lip(M)^*=\mathcal{F}(M)$ and $lip(M)$ is an M-embedded space (see Remark after Theorem 6.6 in \cite{kal}), that is $lip(M)$ is an M-ideal in $Lip(M)$ (see \cite{bw} for the case $M=[0,1]$).  Then we get from \cite{aln} that $lip(M)$ and  $Lip(M)$ satisfy the strong diameter two property. Recall that $lip(M)$ stands for the space of scalar Lipschitz functions on M such that
$$\lim_{\varepsilon\to 0}\sup\left\{\frac{\vert f(x)-f(y)\vert}{d(x,y)},\ x\neq y\in M,\ d(x,y)<\varepsilon\right\}=0.$$

Moreover, in \cite[Theorems 1 and 2]{ivakhno} the author get in the scalar case that $Lip(M)$ has the slice diameter two property whenever $M$ satisfy the same assumptions to the ones of Theorem \ref{teocentral}.

In \cite{cdw} it has been recently proved that $\mathcal F(M)$ contains an isomorphic copy of $\ell_1$ whenever $M$ is an infinite metric space. However, this fact is an easy consequence of Theorem \ref{teocentral}.

\begin{corollary}
Let $M$ be an infinite metric space with a designated origin $0$. Then $\mathcal F(M)$ contains an isomorphic copy of $\ell_1$.
\end{corollary}

\begin{proof}
The corollary follows whenever $M$ satisfies any of the assumptions of Theorem \ref{teocentral}.

In other case, the corollary follows from \cite[Proposition 5.1]{god}.
\end{proof}

\begin{remark} In \cite{cdw} is proved that $\mathcal{F}(M)$ contains a complemented copy of $\ell_1$ whenever $M$ is an infinite metric space. In fact, the results in \cite{cdw} give that $\mathcal{F}(M, X)$ contains too a complemented copy of $\ell_1$, for every Banach space $X$. Indeed, as $Lip(M)$ is isometrically isomorphic to a closed subspace of $Lip(M,X)$, then $Lip(M,X)$ contains an isomorphic copy of $\ell_{\infty}$. Finally, from \cite[Th. 4]{bepe}, we obtain that $\mathcal{F}(M, X)$ contains a complemented copy of $\ell_1$. We want to thank M. Doucha for noticing us about this fact. We would also want to thank M. C\'uth for asking us about the identification of vector-valued Lipschitz-free Banach spaces with a projective tensor product spaces.
\end{remark}












Propositions \ref{metriraropuntodife} and \ref{metrirarodiam2} show that geometry of vector-valued Lipschitz-free Banach spaces does not only depend on underlying scalar Lipschitz-free space but also on the target Banach space. However, two natural questions arise.

\begin{question}\label{estaoctalibrevector}

Let $M$ be a pointed metric space and let $X$ be a non-zero Banach space.

\begin{enumerate}
\item Does Theorem \ref{teocentral} hold without assuming that the pair $(M,X^*)$ has the CPE?
\item Does $\mathcal F(M,X)$ have an octahedral norm whenever $\mathcal F(M)$ does?
\end{enumerate}

\end{question}

Bearing in mind the identification $\mathcal F(M,X)=\mathcal F(M)\widehat{\otimes}_\pi X$, above Question is related to the problem of how octahedrality is preserved by projective tensor products. However, the last one is an open problem recently posed in \cite{llr}.

Finally, we have analysed octahedrality in $\mathcal F(M,X)$ whenever $M$ is a metric space and $X$ is a Banach space. However, we did not get any result about the dual properties (i.e. diameter two properties). More strictly.

\begin{question}

Given $M$ a metric space and $X$ a non-zero Banach space.

Which assumptions do we need over $M$ and $X$ in order to ensure that $\mathcal F(M,X)$ has the slice diameter two property (respectively diameter two property, strong diameter two property)?

\end{question}

Again, not only do we get a partial answer in scalar case but also in vector valued one. Indeed, again by \cite{ikw} we know that $\mathcal F(M)$ has the strong diameter two property whenever $M$ is a metrically convex metric space. Keeping in mind that $\mathcal F(M,X)=\mathcal F(M)\widehat{\otimes}_\pi X$, next Proposition is an inmediate application of \cite[Corollary 3.6]{tensor}.

\begin{proposition}\label{d2penlibre}
Let $M$ be a metric space with a designated origin $0$ and $X$ be a Banach space. If $M$ is metrically convex and $X$ has the strong diameter two property, then $\mathcal F(M,X)$ has the strong diameter two property.

\end{proposition}

Despite above Proposition, there are metric spaces whose free-Lipschitz Banach space fails to have any diameter two property. Indeed, it is well known that $\mathcal F(M)$ has the Radon-Nikodym property whenever $M$ is a totally discrete metric sapce \cite{kal}. Related to the strong diameter two property we can even get vector valued free Lipschitz Banach spaces which fail such property. Indeed, if we consider $X$ a Banach space failing the strong diameter two property and $M$ a totally discrete metric space then $\mathcal F(M,X)=\mathcal F(M)\widehat{\otimes}_\pi X$ does not have the strong diameter two property \cite[Corollary 3.13]{tensor}.

\end{document}